\documentclass[12pt]{article}
\usepackage{pudlatex}
\usepackage{fullpage}
\usepackage[all,cmtip]{xy}
\usepackage{amsmath,amssymb}
\usepackage{xcolor}
\usepackage{turnstile}
\usepackage{authblk}
\usepackage{amsfonts}

\def\EXP{\hbox{\bf EXP}}

\def\lr{\leftrightarrow}
\def\Ra{\Rightarrow}

\def\all{\forall}
\def\ex{\exists}

\def\E{{\mathcal E}}
\def\P{{\mathcal P}}

\def\tfnp{\mathsf{TFNP}}
\def\sat{\mathsf{SAT}}
\def\rfnN{\mathsf{RFN^N_1}}
\def\rfn{\mathsf{RFN_1}}
\def\rf{\mathsf{RFN}}
\def\E{\mathsf{E}}
\def\Co{\mathsf{Co}}
\def\NE{\mathsf{NE}}
\def\NP{\mathsf{NP}}

\def\disconp{\mathsf{DisjCoNP}}

\def\FP{\mathsf{FP}}
\def\con{\mathsf{CON}}
\def\s{\mathsf{S^1_2}}
\def\N{\mathbb{N}}
\def\P{\mathsf{P}}
\def\PP{\mathcal{P}}

\def\H{\mathsf{H}}

\def\UU{\mathsf{U_{\Pi^b_2}}}
\def\UN{\mathsf{U_{NP}}}
\def\UC{\mathsf{U_{CoNP}}}

\def\taut{\mathsf{TAUT}}

\def\Con{\mathsf{Con}}
\def\conN{\mathsf{CON^N}}

\def\EXP{\mathsf{EXP}}
\def\NEXP{\mathsf{NEXP}}

\def\ssat{{\tt Sat}}
\def\staut{{\tt Taut}}
\def\F{\mathsf{F}}  
\def\tauts{\staut_{\Sigma^q_1}}
\def\sats{\ssat_{\Sigma^q_1}}  
\def\disj{\mathsf{Disj}}
\def\V{{\cal V}}
\def\W{{\cal W}}
\def\D{{\sf Dom}}
\def\R{{\sf Rng}}

\title{New relations and separations of conjectures about 
	incompleteness in the finite domain}

\author{Erfan Khaniki\thanks{e.khaniki@gmail.com}}

\affil{Department of Mathematical Sciences\\
	Sharif University of Technology\\
	Tehran, Iran}

\begin{document}
	\maketitle
	\begin{abstract}
		Our main results are in the following three sections:
		\begin{enumerate}
			\item We prove new relations between proof complexity conjectures that are discussed in \cite{pu18}.
			\item We investigate the existence of p-optimal proof systems for $\taut$, assuming the collapse of $\cal C$ and $\sf N{\cal C}$ (the nondeterministic version of $\cal C$) for some new classes $\cal C$ and also prove new conditional independence results for strong theories, assuming nonexistence of p-optimal proof systems.
			\item We construct two new oracles $\V$ and $\W$. These two oracles imply several new separations of proof complexity conjectures in relativized worlds. Among them,
			we prove that existence of a p-optimal proof system for $\taut$ and existence of a
			complete problem for $\tfnp$ are independent of each other in relativized worlds which was not known before.
		\end{enumerate}
	\end{abstract}
	\section{Introduction}
	Proof complexity is a branch of mathematical logic and computational complexity which
	is concerned with the length of proofs of tautologies in different proof systems. The main goal
	is to develop techniques to prove lower bounds for all propositional proof systems, which
	would entail $\sf NP \neq CoNP$. In \cite{pu18}, the main conjectures of proof
	complexity, for example, the existence of p-optimal proof systems or the existence of a complete problem
	$\tfnp$ with respect to the poly time reductions, are investigated from the point of view of logical strength to prove
	these statements. For every one of the main conjectures of proof complexity, an
	equivalent conjecture is proposed in terms of unprovability of statements in strong enough
	theories. Thus, it creates the possibility to use mathematical logic methods to attack these
	conjectures. The logical methods, e.g., a version of forcing used in \cite{kr11}, indeed were
	successful in some very important results in proof complexity. See \cite{kp89,aj94,bkkk17,ri97}.
	
	This paper contains three sections. In the first section, we prove new relations between
	conjectures of \cite{pu18}. In section 2, we investigate the existence of p-optimal proof systems
	$\taut$, assuming the collapse of $\cal C$ and $\sf N{\cal C}$ for some new classes $\cal C$. This investigation leads to
	a generalization of the conjectures in \cite{pu18} to use reductions in the complexity classes of quasipolynomial or subexponential time computable functions. These generalized conjectures have
	the same relation among each other like the relations between conjectures of \cite{pu18}. We prove
	new relations between collapsing complexity classes and the existence of the optimal proof
	systems and we show that proving the collapse of some complexity classes constructively
	implies the existence of optimal proof systems for $\taut$. In addition, we prove for every
	strong enough theory $T$, there is a language $L\in {\sf N}{\cal C}$, such that for every natural definition of
	a language $L' \in {\cal C}$, $T \not\vdash L = L'$ for some classes $\cal C$, assuming that there is no p-optimal proof
	system. In section 3, we construct two new oracles. Relative to the first oracle, a p-optimal proof
	system for $\taut$ exists, but the class of disjoint $\sf CoNP$ problems does not have complete
	problems with respect to poly time functions. Relative to the second oracle, $\tfnp$ is
	equal to $\sf FP$, but length optimal proof systems do not exist. These two oracles imply several
	new separations of conjectures of \cite{pu18} in relativized worlds.
	\section{Preliminaries}
	Following the notation of \cite{pu18}, we use first order theories of arithmetic in a fixed language. The language is the standard language of bounded arithmetic, which is $$\mathcal{L}_{BA}=\{0, S,+,\cdot,|x|,\lfloor x/2\rfloor, x\# y\}.$$ The intended meaning of the $\lfloor x/2\rfloor$ is clear. The meaning of the $|x|$ is $\lceil \log_2(x+1)\rceil$. $x\# y$ is interpreted as $2^{|x|\cdot |y|}$. 
	
	A sharply bounded quantifier is of the form $Qx<|t|,Q\in\{\all,\ex\}$. The class of bounded formulas $\Sigma^b_n$, $\Pi^b_n$, $n \geq 1$ is defined by counting alternations of bounded quantifiers while ignoring sharply bounded
	quantifiers (see \cite{bu86}). The class of $\Delta^b_n$ formulas is the class of $\Sigma^b_i$ formulas that have an equivalent $\Pi^b_i$ definition. The theory $\s$ is consists of basic axioms defining the usual properties of the function symbols and by induction axioms 
	$$\phi(0)\land\all x(\phi(\lfloor x/2\rfloor)\to\phi(x))\to \all x\phi(x)$$ for all $\Sigma^b_1$ formulas. $\s$ is the base theory in provability with respect to the bounded arithmetic hierarchy like ${\bf I}\Sigma_1$ with respect to Peano arithmetic. One of the main properties of $\s$ is that $\Sigma^b_1$ definable functions of $\s$ are poly time computable. Additionally, all of the poly time computable functions are $\Delta^b_1$ in $\s$ (A $\Sigma^b_1$ formula $\phi$ is $\Delta^b_1$ in $T$ iff there exists a $\Pi^b_1$ formula $\psi$ such that $T\vdash \phi \equiv \psi$). For more information about bounded arithmetics see \cite{bu86}.
	
	Let $\mathcal{T}$ be the set of all consistent first order theory $\s\subseteq T$ in $\mathcal{L}_{BA}$ such that the set of axioms of $T$ is poly time decidable. The main objects of concern in \cite{pu18} are unprovability and provability results with respect to the members of $\mathcal{T}$. \cite{pu18} translates the well-known conjectures in complexity theory and proof complexity to unprovability statements about members of $\mathcal{T}$.
	
	Next, we will explain notations and definitions for proof complexity conjectures and their translation in \cite{pu18}.
	\subsection{$\tfnp$ class}
	$\tfnp$ or Total $\NP$ search problem is the class of true $\all\Sigma^b_1$ sentences. More formally, a total $\NP$ search problem is defined by the pair $(p,R)$ such that:
	\begin{enumerate}
		\item $p(x)$ is a polynomial,
		\item $R(x,y)$ is a poly time computable relation ($\Delta^b_1$ in $\s$),
		\item $\N\models \all x \exists y(|y|\leq p(|x|)\land R(x,y))$. 
	\end{enumerate}
	For comparing the complexity of $\tfnp$ problems, reductions are defined as follows.
	\begin{definition}
		Suppose $P$ and $Q$ are in $\tfnp$. We say $P$ is polynomially reducible to $Q$ if the search problem $P$ can be solved in polynomial time using an oracle that gives the answers to the search problem $Q$.
	\end{definition}
	There are different classes of $\tfnp$ which are defined by reductions in the seminal paper \cite{jpy88}. These classes are of the form of {\em all $\tfnp$ problems that are reducible to a $\tfnp$ problem $P$}.
	Another way to compare the complexity of $\tfnp$ problems is by measuring how strong axioms are needed to prove a search problem is total. This approach has reductions implicitly in it. The next definition formalizes this notion which is defined in \cite{pu18}.
	\begin{definition}
		Suppose $T$ is in $\mathcal{T}$. We say $(p,R)$ is provably total in $T$ or $(p,R)\in \tfnp(T)$ iff there exists a pair $(q,\phi)$ such that:
		\begin{enumerate}
			\item $q$ is a polynomial,
			\item $\phi(x,y)$ is $\Delta^b_1$ in $\s$,
			\item $\mathbb{N}\models \all x,y((|y|\leq p(|x|) \land R(x,y))\equiv (|y|\leq q(|x|) \land \phi(x,y)))$,
			\item $T\vdash \all x \exists y(|y|\leq q(|x|) \land \phi(x,y))$.
		\end{enumerate}
		Also, we define $\tfnp^*(T)$ as the class of all $\tfnp$ problems that is reducible to a problem in $\tfnp(T)$.
	\end{definition}
	For many bounded arithmetic $T\in{\cal T}$ such as Buss's bounded arithmetics, $\tfnp(T)$ is characterized. Actually, $\tfnp(T)$ for a bounded arithmetic theory $T\in \mathcal{T}$ is a measurement of the strength of the bounded arithmetic $T$, like the provably total recursive functions for strong theories. The following theorem shows the relationship between the strength of reduction and provability.
	\begin{theorem}\label{t2.1}
		(\cite{pu18}) The following statements are equivalent:
		\begin{enumerate}
			\item There exists a problem $(p,R)\in\tfnp$ that is complete, with respect to the polynomial reductions for class $\tfnp$,
			\item There exists $T\in\mathcal{T}$ such that $\tfnp^*(T)=\tfnp$.
		\end{enumerate}
	\end{theorem}
	The main conjecture about $\tfnp$ class is that it does not have a complete problem with respect to polynomial reductions. We will show this conjecture by $\tfnp_c$. 
	\subsection{Proof systems}
	Following the definition of Cook-Reckhow, a proof system for set $C\subseteq \mathbb{N}$ is a poly time computable function $P:\mathbb{N}\to \mathbb{N}$ (the graph of $P$ is $\Delta^b_1$ in $\s$) such that $\R(P)=C$. We assume that different objects such as formulas, proofs, etc. are coded in a natural way in binary strings, hence every binary code $x$ can be shown by a natural number with binary expansion $1x$, which we will denote by $\llcorner x\lrcorner$. To code a sequence of finite binary strings $x_1$ to $x_n$ that is shown by $\left<x_1,...,x_n\right>$, we use the following coding $x^*_1x^*_2...x^*_{n-1}x_n$, for which a binary string $z$, $z^*$ is obtained from $z$ by doubling its digits and appending the string $01$ at the end of it. Note that we can use the same coding schema for coding a finite sequence of natural numbers. By this explanation, we can define proof systems for different sets, such as propositional tautologies ($\taut$) or satisfiable propositional formulas ($\sat$). By length of an object (formulas, proofs,...) with the natural number $n$ as its code, we mean $|n|$. For every object $A$, we will use the notation $\ulcorner A \urcorner$ to show the numerical code of $A$.
	
	A proof system $P$ for set $C$ is poly bounded iff there exists a polynomial $q(x)$ such that for every $n\in C$, there exists a proof $\pi\in\mathbb{N}$ such that $P(\pi)=n$ and $|\pi|\leq q(|n|)$. One of the most important conjectures in proof complexity is the nonexistence of a poly bounded proof system for $\taut$. In terms of complexity theory language, this conjecture is equivalent to $\NP\not=\Co\NP$. Another concept that is weaker than poly boundedness is optimality. The following definition formalizes the components of this concept.
	\begin{definition}
		Suppose $P$ and $Q$ are proof systems for set $C$. We say that $P$ non-uniformly p-simulates $Q$ iff there exists a polynomial $h(x)$ such that: $$\all \pi\in\mathbb{N},\all n\in C(Q(\pi)=n\to \exists \pi'\in\mathbb{N}(|\pi'|\leq h(|\pi|)\land P(\pi')=n))$$
		We say that $P$ p-simulates $Q$ iff there exists a poly time function $f$ such that:$$\all \pi\in\mathbb{N},\all n\in C(Q(\pi)=n\to P(f(\pi))=n)$$
	\end{definition} 
	Normally, non-uniform p-simulation is called simulation in the literature, but because we will generalize these concepts to bigger complexity classes, we named it in this way to make it distinguishable with generalized cases.
	
	We call a proof system $P$ for set $C$ is (non-uniform) p-optimal iff for every proof system $Q$ for set $C$, $P$ (non-uniform) p-simulates $Q$. One of the main conjectures about (non-uniform) p-optimality is that there is no (non-uniform) p-optimal proof system for $\taut$. We will show these conjectures with $\con$ and $\conN$ in which $\sf N$ stands for nonuniform. Another important conjecture about p-optimality is that there is no p-optimal proof system for $\sat$, which we call $\sat_c$. To translate these conjectures to provability and unprovability of theories in $\mathcal{T}$ we need to define some machinery. Note that for every $T\in\mathcal{T}$, because the axioms of $T$ are poly time decidable, there exists a poly time computable relation $Pr_T(x,y)$ in which it is true iff $x$ is code of a $T$-proof in the usual Hilbert style calculi of a formula in $\mathcal{L}_{BA}$ with code $y$. One of the important properties of $Pr_T(x,y)$ is the following theorem.
	\begin{theorem}(\cite{bu86})
		For every $T\in{\cal T}$, every $\Sigma^b_1$ formula $\phi(x)$, there exists a polynomial $p(x)$ such that $T\vdash\all x(\phi(x)\to\ex y(|y|\leq p(|x|)\land Pr_T(y,\ulcorner\phi(\dot{x})\urcorner))$.
	\end{theorem}
	Note that for every nonempty set $C\subseteq \mathbb{N}$, $C$ has a proof system iff $C$ is recursively enumerable. Suppose $C\subseteq\mathbb{N}$ is a nonempty recursively enumerable set. Let $\phi_C(x)$ be a $\Sigma_1$ formula in $\mathcal{L}_{BA}$ defining $C$. To define a proof system for $\phi_C(x)$ from a theory $T\in\mathcal{T}$, we need to define a natural number in $\mathcal{L}_{BA}$ in an efficient way. The following definition gives us an efficient way of defining the numerals.
	\begin{definition}
		
		$ $
		
		$\bar{n}=\begin{cases}
		0 & n=0\\
		SS0\cdot\bar{k}&n=2k\\
		S(SS0\cdot\bar{k})&n=2k+1
		\end{cases}$
		
		Note that the coded version of $\bar{n}$ needs $O(\log_2 n)$ bits. Additionally, the notation $\ulcorner \phi(\dot{n})\urcorner$ for formula $\phi(x)$ in $\mathcal{L}_{BA}$ is a poly time computable function such that it outputs the code of formula $\phi(\bar{n})$.
	\end{definition}
	Suppose $a$ is in $C$. Now we define the proof system $P^C_T$ associated with $T$ for $C$ as follows: 
	\begin{enumerate}
		\item Given $\pi$, if $\mathbb{N}\models Pr_T(\pi, \ulcorner \phi_C(\dot{n})\urcorner)$ for some $n$, then outputs $n$,
		\item otherwise outputs $a$.
	\end{enumerate}
	Let $Con_T(n)$ be the formula $\all x(|x|\leq n \to \neg Pr_T(x,\ulcorner\bot\urcorner
	))$. Using above notations and definitions we can express theorems that show the relationship between optimality of proof systems and provability in members of $\mathcal{T}$.
	\begin{theorem}\label{t2.2}
		(\cite{kp89})
		The following statements are equivalent:
		\begin{enumerate}
			\item There exists a nonuniform p-optimal proof system for $\taut$,
			\item There exists $T\in\mathcal{T}$ such that for every $S\in\mathcal{T}$, the shortest $T$-proofs of $Con_S(\bar{n})$ is bounded by a polynomial in $n$. 
		\end{enumerate}
	\end{theorem}
	To work with propositional tautologies and satisfiable formulas we use the poly time computable relation $\ssat(x,y)$, which means the propositional formula with code $x$ is satisfiable in assignment with code $y$. Also, we use $\Pi^b_1$ notation $\staut(x):=\all y(y\leq x\to \ssat(x,y))$ to define propositional tautologies. In order to work with $\all\Pi^b_1$ and $\all\Pi^b_1(\alpha)$ sentences as a family of propositional tautologies, we use the usual translation of $\all\Pi^b_1$ sentences, and Paris-Wilkie translation of $\all\Pi^b_1(\alpha)$ sentences as defined in \cite{bu97}.
	\begin{theorem}\label{t2.7}
		(\cite{kp89})
		The following statements are equivalent:
		\begin{enumerate}
			\item There exists a p-optimal proof system for $\taut$,
			\item There exists $T\in\mathcal{T}$ such that for every $S\in\mathcal{T}$, there exists a poly time computable function $h$ that for every $n$, $h(n)$ is a $T$-proof of $Con_S(\bar{n})$.
			\item There exists $T\in\mathcal{T}$ such that for every proof system $P$ for $\taut$, there exists a poly time formalization $P'(x,y)$ of relation $P(x)=y$ that $$T\vdash \all x,y(P'(x,y)\to\staut(y)).$$ 
		\end{enumerate}
	\end{theorem}
	The following theorem gives a translation of the nonexistence of the p-optimal proof system for $\sat$.
	\begin{theorem}\label{t2.8}
		(\cite{pu18})
		The following statements are equivalent:
		\begin{enumerate}
			\item There exists a p-optimal proof system for $\sat$,
			\item There exists $T\in\mathcal{T}$ such that for every proof system $P$ for $\sat$, there exists a poly time formalization $P'(x,y)$ of relation $P(x)=y$ that $$T\vdash \all x,y(P'(x,y)\to\ex z(z<y\land \ssat(y,z))).$$ 
		\end{enumerate}
	\end{theorem}
	\subsection{Disjoint $\NP$ pairs, disjoint $\Co\NP$ pairs}
	The concept of disjoint $\NP$ pairs and disjoint $\Co\NP$ pairs are discussed in \cite{pu18} to define stronger conjectures than $\tfnp_c$ and $\conN$. A pair of $(\Co)\NP$ languages $(U,V)$ is a disjoint $(\Co)\NP$ pair iff $U\cap V=\varnothing$. We will show this class of pairs by $\mathsf{Disj}(\Co)\NP$. In order to compare the complexity of disjoint $(\Co)\NP$ pairs, the reductions are defined as follows:
	\begin{definition}
		Suppose $(U_0,U_1)$ and $(U'_0,U'_1)$ are disjoint $(\Co)\NP$ pairs. We say $(U_0,U_1)$ is polynomial reducible to $(U'_0,U'_1)$ iff there exists a poly time computable function $f$ such that for $i\in\{0,1\}$:
		$$\all n \in\mathbb{N}(n\in U_i\to f(n)\in U'_i)$$
	\end{definition}
	Again, another way to compare the complexity of disjoint $(\Co)\NP$ pairs is by measuring how strong axioms are needed to prove such a pair is disjoint. The next definition formalizes this notion.
	\begin{definition}
		Suppose $T$ is in $\mathcal{T}$. We say $(\Co)\NP$ pair $(U_0,U_1)$ is provably disjoint in $T$ or $(U_0,U_1)\in \mathsf{Disj}(\Co)\NP(T)$ iff there exists a ($\Pi^b_1$) $\Sigma^b_1$ pair $(\phi_0,\phi_1)$ such that:
		\begin{enumerate}
			\item $\mathbb{N}\models\all x(x\in U_i \equiv \phi_i(x)),i\in\{0,1\}$,
			\item $T\vdash \all x(\neg\phi_0(x)\lor \neg\phi_1(x))$.
		\end{enumerate}
	\end{definition}
	Like theorem \ref{t2.1}, the following theorem shows the relationship between the strength of reduction and provability.
	\begin{theorem}\label{t2.5}
		(\cite{pu18}) The following statements are equivalent:
		\begin{enumerate}
			\item There exists a pair $(U,V)\in\disj(\Co)\NP$ that is complete with respect to the polynomial reductions for class $\disj(\Co)\NP$,
			\item There exists $T\in\mathcal{T}$ such that $\disj(\Co)\NP(T)=\disj(\Co)\NP$.
		\end{enumerate}
	\end{theorem}
	The main conjecture about disjoint $(\Co)\NP$ pairs is that it does not have a complete problem with respect to polynomial reductions. We will show this conjecture by $\disj(\Co)\NP_c$.
	\subsection{A finite reflection principle}
	A finite reflection principle for $\Sigma^b_1$ formulas is defined in \cite{pu18} to propose a conjecture that connects defined conjectures in this section. To define the conjecture, we need the following theorem.
	\begin{theorem}(\cite{ph93})
		For every $i\geq 1$ there exists a $\Sigma^b_i$ formula $\mu_i$ such that for every $\Sigma^b_i$ formula $\phi(x)$ there exists natural number $e$ and polynomial $p$ such that:
		$$\s\vdash \all x,y(|y|\geq p(|x|)\to(\mu_i(\bar{e},x,y)\equiv \phi(x)))$$
	\end{theorem}
	The finite reflection principle is defined as follows:
	\begin{definition}
		For every $T\in\mathcal{T}$, $n\in\mathbb{N}$, the $\Sigma^b_1\mathsf{RFN}_T(\bar{n})$ is defined by
		$$\all e,u,x,z(|e|,|u|,|x|,|z|\leq \bar{n}\land Pr_T(u,\ulcorner\mu_1(\dot{e},\dot{x},\dot{z})\urcorner)\to \mu_1(e,x,z)).$$
	\end{definition}
	The following conjectures are defined in \cite{pu18}:
	\begin{enumerate}
		\item $\rfnN$: For every $T\in\mathcal{T}$, there exists $S\in\mathcal{T}$ such that the $T$-proofs of $\Sigma^b_1\rf_S(\bar{n})$ are not polynomially bounded in $n$.
		\item $\rfn$: For every $T\in\mathcal{T}$, there exists $S\in\mathcal{T}$ such that the $T$-proofs of $ \Sigma^b_1\rf_S(\bar{n})$ can not be constructed in polynomial time.
	\end{enumerate}
	The following figure shows the relation between conjectures of this section. For more information about the proof of these relations see \cite{pu18}.
	$$\xymatrix@C=5mm{            
		{\disj\NP_c}\ar[d]&&{\disconp_c}\ar[d]\\
		{\conN}\ar[d]\ar[rd]&&{\tfnp_c}\ar[d]\\
		{\rfnN} \ar[rd]\ar[d]&{\con}\ar[d] & {\sat_c}\ar[ld]\\
		{\NP\not=\Co\NP}\ar[rd]&{\rfn}\ar[d]\\
		&{{\sf P}\not=\NP}
	}$$
	\begin{figure}[h!]
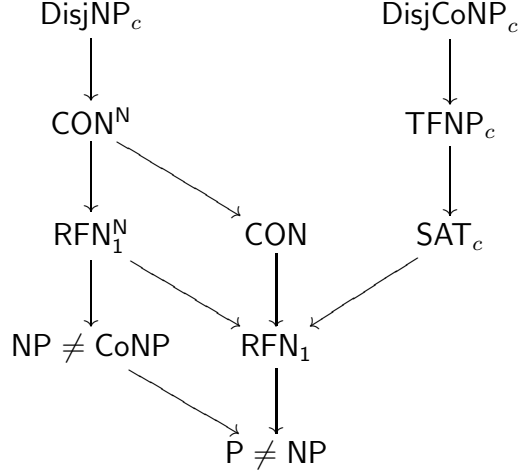

		\caption{Relations between conjectures}
		\label{fig}
	\end{figure}
	\section{Incompleteness in the finite domain}
	\subsection{Some observations on $\tfnp$ class}
	As we see in the previous section, the logical equivalent conjectures that are discussed are of the following form:
	
	{\em For every $T\in\mathcal{T}$ there exists some sentence $\phi$ that does not have $T$-proof with some properties.}
	
	The above form works for all of the conjectures that we discussed, except for $\tfnp_c$. The logical form of $\tfnp_c$ conjecture uses $\tfnp^*(T)$ instead of $\tfnp(T)$. Here we want to investigate what happens if we use $\tfnp(T)$. This new conjecture, which we call $\tfnp^w_c$ is weaker than $\tfnp_c$. The next proposition shows that it is stronger than $\sat_c$.
	\begin{proposition}
		If for every $T\in\mathcal{T}$ we have $\tfnp(T)\not=\tfnp$, then there is no p-optimal proof system for $\sat$.
	\end{proposition}
	\begin{proof}
		Suppose $P$ is a p-optimal proof system for $\sat$. Define $T:=\s+\all x \ex y \ssat(P(x),y)$. Let $(p,R)$ be a $\tfnp$ problem and $(q,\phi)$ be one of its formalizations. Suppose $F$ is a proof system for $\sat$. Let $\theta_n$ be the usual propositional translation polynomial time relation $|y|\leq q(|\bar{n}|)\land \phi(\bar{n},y)$. The proof system $P_\phi$ for $\sat$ is defined as follows:
		$$P_\phi(x)=\begin{cases}F(n) & x=2n\\\theta_n& x =2n+1\end{cases}$$
		Because $P$ is a p-optimal proof system, there exists a poly time function $h$ such that $\N\models \all x(P(h(x))=P_{\phi}(x))$. This implies that $$\N\models\all x,y\big((|y|\leq q(|x|)\land\phi(x,y)) \equiv \ssat(P(h(2x+1)),f(y))\big)$$ for some poly time function $f$,  hence $\ssat(P(h(2x+1)),f(y))$ is another formalization of $(p,R)$. Note that by definition of $T$ we have $T\vdash \all x\ex y \ssat(P(h(2x+1)),f(y))$ which means $(p,R)\in\tfnp(T)$.
	\end{proof}
	
	We can not prove that $\tfnp^w_c$ implies $\tfnp_c$, but one way to show that the latter conjecture is probably stronger is to find a $T\in\mathcal{T}$ such that $\tfnp(T)\not=\tfnp^*(T)$. It is conjectured that such a $T$ exists, but we observed that existence of such a $T$ implies $\tfnp\not=\FP$, hence proving this conjecture unconditionally is hard. We need the following lemma to prove the previous implication.
	\begin{lemma}\label{l2.3}
		$\tfnp(\s)={\sf FP}$.
	\end{lemma} 
	\begin{proof}
		By the fact that $\Sigma^b_1$ definable functions of $\s$ is poly time computable we get $\tfnp(\s)\subseteq \FP$, so it is sufficient to prove $\FP\subseteq \tfnp(\s)$. Let $(p,R)$ be a $\tfnp$ problem which can be solved by the poly time function $f$. Let $\phi$ be the $\Delta^b_1$ formalization of $f$ in $\s$. Additionally, let $(q,\psi)$ be a formalization of $(p,R)$. Note that $(q, \psi \lor \phi)$ is a formalization of $(p,R)$ and also $\s\vdash \all x\ex y(|y|\leq q(|x|)\land (\psi(x,y)\lor \phi(x,y))$, hence $(p,R)\in\tfnp(\s)$, which implies $\FP\subseteq \tfnp(\s)$.
	\end{proof}
	\begin{corollary}
		If there exists $T\in\mathcal{T}$ such that $\tfnp(T)\not=\tfnp^*(T)$, then $\tfnp\not=\FP$.
	\end{corollary}
	\begin{proof}
		Suppose $\tfnp$ is equal to $\FP$, hence for every $T\in\mathcal{T}$, $\tfnp(T)\subseteq \FP$, which implies $\tfnp^*(T)\subseteq \FP^\FP=\FP$. Also, by definition of $T$ and lemma \ref{l2.3}, $\FP=\tfnp(\s)\subseteq \tfnp(T)$, hence $\tfnp(T)=\tfnp^*(T)=\FP$, which completes the proof.
	\end{proof}
	\subsection{On proof systems and $\rfn$ conjecture}
	As we have noted, every conjecture that is discussed in the previous section has two formalizations, one in terms of proof complexity notations, and one in terms of incompleteness in the finite domain notations, except $\rfn$ and $\rfnN$. Here we want to show that these conjectures have equivalent forms in terms of optimal proof systems for $\Sigma^q_1$-$\taut$. $\Sigma^q_i$ ($\Pi^q_i$) propositional formulas are quantified propositional formulas and defined like the hierarchy of bounded formulas in $\mathcal{L}_{BA}$. The next theorem is similar to theorems \ref{t2.7} and \ref{t2.8} for $\con$ and $\conN$. 
	\begin{theorem}\label{t3.4}
		$ $
		
		\begin{enumerate}
			\item 
			The following statements are equivalent:
			\begin{enumerate}
				\item For every $T\in\mathcal{T}$, there exists $S\in\mathcal{T}$ such that the $T$-proofs of $\Sigma^b_1\rf_S(\bar{n})$ are not polynomially bounded in $n$.
				\item $\Sigma^q_1$-$\taut$ does not have a nonuniform p-optimal proof system.
			\end{enumerate}
			\item 
			The following statements are equivalent:
			\begin{enumerate}
				\item For every $T\in\mathcal{T}$, there exists $S\in\mathcal{T}$ such that the $T$-proofs of $ \Sigma^b_1\rf_S(\bar{n})$ can not be constructed in polynomial time.
				\item $\Sigma^q_1$-$\taut$ does not have a p-optimal proof system.
				\item For every theory $T\in\mathcal{T}$, there exists a proof system $P$ for $\Sigma^q_1$-$\taut$ such that $T$ does not prove the soundness of any formalization of $P$.
			\end{enumerate}
		\end{enumerate}
	\end{theorem}
	\begin{proof}
		Here we prove the second part. The proof of the first part is similar.
		\begin{itemize}
			\item [$(a)\Ra(b)$.] Suppose $(b)$ is false. Let $P$ be a p-optimal proof system for $\Sigma^q_1$-$\taut$. Let $T:=\s + \all \pi\tauts(P(\pi))$ in which $\tauts$ is the $\Pi^b_2$ formula that checks whether a $\Sigma^q_1$ propositional formula is true or not. Let $S\in\mathcal{T}$. Note that for every $n\in\mathbb{N}$, the translation of $\Sigma^b_1\rf_S(\bar{n})$ is $\Sigma^q_1$ formula $\theta_n$ such that $\s\vdash \Sigma^b_1\rf_S(\bar{n})\equiv \tauts(\theta_n)$ and this proof can be constructed in poly time (see \cite{bu97} for the propositional case.) $(*)$. Let $P'$ be a proof system defined as follows:$$P'(x)=\begin{cases}\theta_n & x = \theta_n\text{ for some $n$}\\P(x) & \text{o.w.}\end{cases}$$
			Let $f$ be the poly time function such that $P(f(\pi))= P'(\pi)$ for every $\pi\in \mathbb{N}$. Note that for every $n\in\mathbb{N}$, the proof of $\s\vdash P(f(\theta_n))=\theta_n$ can be constructed in poly time, therefore by soundness of $P$ which is provable in $T$ and $(*)$, the proof of $T\vdash\Sigma^b_1\rf_S(\bar{n})$ for every $n\in\mathbb{N}$ can be constructed in poly time too.
			\item [$(b)\Ra(c)$.] Suppose $(c)$ is false. Let $T\in\mathcal{T}$ be a theory that falsifies $(c)$. We want to prove that $P^{\Sigma^q_1}_T$ is p-optimal . Let $P'$ be a proof system and $P''$ be one of its formalizations such that $T\vdash \all \pi\tauts(P''(\pi))$.  Note that there exists a poly time function $f$ such that $$T\vdash \all \pi,\phi(P''(\pi)=\phi\to Pr_T(f(\pi,\phi),\ulcorner P''(\dot{\pi})=\dot{\phi}\urcorner)),$$
			hence there exists a poly time function $h$ such that $P''(\pi)=P^{\Sigma^q_1}_T(h(\pi))$, for all $\pi\in \mathbb{N}$.
			\item [$(b)\Ra (a)$.] Suppose $(a)$ is false. Let $T\in\mathcal{T}$ be a theory that witnesses this fact. We show that $P^{\Sigma^q_1}_T$ is p-optimal. Let $\sats(\phi,v)$ be the $\Sigma^b_1$ formula that can check the satisfiability of $\Sigma^q_1$ propositional formulas. Define $T':=\s + \all \pi,v\sats(P(\pi),v)$. If $P(\pi_\psi)=\psi$, then we can find a proof $\pi'$ in poly time such that $P^{\Sigma^q_1}_{T'}(\pi')=\psi$ (*). Note that there exists a poly time function $f$ such that $$\mathbb{N}\models \all \pi,v,\phi(|v|\leq|\phi|\land P^{\Sigma^q_1}_{T'}(\pi)=\phi\to P^{\Sigma^q_1}_{T'}(f(\pi,v))=\phi\big[v/\vec{p}\big]).$$
			Let $T'':=\s+\all \pi,v,\phi(|v|\leq|\phi|\land P^{\Sigma^q_1}_{T'}(\pi)=\phi\to P^{\Sigma^q_1}_{T'}(f(\pi,v))=\phi\big[v/\vec{p}\big])$. Note that $T$ falsifies $\rfn$, hence $P_T$ is a p-optimal proof system for $\taut$, this means $P_T$ p-simulates $P_{T''}$(**). Note that propositional translations of $$\all \pi,v,\phi(|v|\leq|\phi|\land P^{\Sigma^q_1}_{T'}(\pi)=\phi\to P^{\Sigma^q_1}_{T'}(f(\pi,v))=\phi\big[v/\vec{p}\big])$$ have short proofs in $P_{T''}$ and these proofs can be constructed in poly time, hence by (*) and (**) we can find a $T'$-proof $\pi''$ of $\all v(|v|\leq |\psi|\to P^{\Sigma^q_1}_{T'}(f(\pi',v),\psi[v/\vec{p}]))$ in poly time, therefore by constructing a $\Sigma^b_1\rf_{T'}(n)$ for some suitable $n$ which is polynomial in size of $\psi$, we can find a proof $\pi^*$ such that $P^{\Sigma^q_1}_{T}(\pi^*)=\psi$. So $P^{\Sigma^q_1}_T$ is p-optimal for $\Sigma^q_1$-$\taut$.
			\item [$(c)\Ra (b)$.] Suppose $(b)$ is false. Let $T\in{\cal T}$ be a theory that witnesses this fact. Thus, the theory $\s+ \all\pi\tauts P^{\Sigma^q_1}_T(\pi)$ falsifies $(c)$.
		\end{itemize}
	\end{proof}

	Note that the previous theorem can be generalized for finite reflection principle conjectures for $\Sigma^b_i$ formulas, as ${\sf RFN}_i$.
	
	By looking at figure \ref{fig}, we observe that the upper conjectures are stronger than those that are behind them and it is not known whether an opposite implication can be proved, i.e. a weak conjecture implies a stronger one. The next theorem shows a kind of opposite implication. In terms of defined notations the next theorem shows that $\rfn$ implies $\con \lor \sat_c$.
	\begin{theorem}
		At least one of the following statements is true:
		\begin{enumerate}
			\item There is no p-optimal proof system for $\sat$,
			\item There is no p-optimal proof system for $\taut$,
			\item There exists a $T\in\mathcal{T}$ such that for every $S\in\mathcal{T}$, the $T$-proofs of $ \Sigma^b_1\rf_S(\bar{n})$ can be constructed in polynomial time.
		\end{enumerate}
	\end{theorem}
	\begin{proof}
		Suppose $(1)$ and $(2)$ are false. Let $T\in\mathcal{T}$ be the theory that falsifies $(1)$ and $(2)$ simultaneously. Suppose is $S$ in $\mathcal{T}$. We want to show that there exists a poly time function $h$ such that for every $\Sigma^q_1$ formula $\phi$ and every $S$-proof $\pi$ of $\all u(|u|\leq |\phi|\to\sats(\phi,u))$, $h(\pi)$ is a $T$-proof of $\all u(|u|\leq |\phi|\to\sats(\phi,u))$. Hence $P^{\Sigma^q_1}_T$ is p-optimal and by theorem \ref{t3.4}, $\neg\rfn$. Note that there exists a poly time function $f$ such that $$\mathbb{N}\models \all \pi,v,\phi(|v|\leq|\phi|\land P^{\Sigma^q_1}_S(\pi)=\phi\to P^{\Sigma^q_1}_S(f(\pi,v))=\phi\big[v/\vec{p}\big]).$$ Suppose $P^{\Sigma^q_1}_S(\pi_\psi)=\psi$ for a $\Sigma^q_1$ formula $\psi$, hence we can find a short $T$-proof of $P^{\Sigma^q_1}_S(\pi_\psi)=\psi$ in poly time $(*)$.
		
		Define $S':=\s+\all \pi,v,\phi(|v|\leq|\phi|\land P^{\Sigma^q_1}_S(\pi)=\phi\to P^{\Sigma^q_1}_S(f(\pi,v))=\phi\big[v/\vec{p}\big])$. Because $T$ falsifies $(2)$ and $S'$ has a short proof of translation of $$\all v(|v|\leq|\psi|\land P^{\Sigma^q_1}_S(\pi_\psi)=\psi\to P^{\Sigma^q_1}_S(f(\pi_\psi,v))=\psi\big[v/\vec{p}\big]),$$ we can find a short $T$-proof of translation of it in poly time. Therefore by $(*)$ we get a $T$-proof of $\all v(|v|\leq|\psi|\to P^{\Sigma^q_1}_S(f(\pi_\psi,v))=\psi\big[v/\vec{p}\big])$ $(**)$. Note that $\psi\big[v/\vec{p}\big]$ does not have free variables, hence there exists a poly time function $g$ such that $T$ has a short proof of $$\all v(|v|\leq|\psi|\to\tauts(\psi\big[v/\vec{p}\big])\equiv \ex u\ssat(g(\psi\big[v/\vec{p}\big]),u)).$$
		 Hence by $(**)$ and by the fact that $T$ proves the $\sat$ proof system defined from $S$ (for some formalization of it) is sound (because $T$ falsifies $(1)$ ), a $T$-proof of $\all v(|v|\leq|\psi|\to \sats(\psi,v))$ can be constructed in poly time.
	\end{proof}
	\section{Nondeterministic vs deterministic computations and existence of optimal proof systems}
	In this section, we investigate the relationship between the equality of nondeterministic and deterministic computation and the existence of optimal proof systems. The trivial case is $\P=\NP$ that implies the existence of poly time computable proofs for $\taut$. The first step in this direction was done in \cite{kp89}. They showed that $\E=\NE$ implies existence of p-optimal proof systems for $\taut$. Latter, It was shown in \cite{mt98} that the condition ${\sf EE}={\sf NEE}$ is sufficient. This phenomenon was investigated further in \cite{bg98} by defining the fat and slim complexity classes and proving the following results about them:
	\begin{enumerate}
		\item \begin{enumerate}
			\item For every slim class ${\cal C}$, ${\cal C}={\sf CoN}{\cal C}$ implies the existence of a nonuniform p-optimal proof system for $\taut$.
			\item For every slim class ${\cal C}$, $\sf N{\cal C}={\sf CoN}{\cal C}$ implies the existence of a p-optimal proof system for $\taut$.
		\end{enumerate}
		\item \begin{enumerate}
			\item For every fat class ${\cal C}$, there exists an oracle $A$ such that ${\cal C}^A={\sf CoN}{\cal C}^A$, but there is no p-optimal proof system for $\taut^A$.
			\item For every fat class ${\cal C}$, there exists an oracle $A$ such that ${\sf N}{\cal C}^A={\sf CoN}{\cal C}^A$, but there is no nonuniform p-optimal proof system for $\taut^A$.
		\end{enumerate}
	\end{enumerate}
	First of all, we prove a similar sufficient condition for the existence of nonuniform and uniform p-optimal proof system for $\Sigma^q_1$-$\taut$. Note that by theorem \ref{t3.4}, the existence of such proof system is equivalent to $\neg\rfnN$ and $\neg\rfn$, respectively. It is shown in \cite{pu18} that $\rfnN$ implies $\NP\not=\Co\NP$. The next proposition strengthens this result. To state the next proposition, we need to define $k$'th Exponential Time Hierarchy.
	\begin{definition}
		Define the following functions inductively:
		\begin{enumerate}
			\item $|x|_n=\begin{cases}|x|_0=x \\|x|_{n+1}=||x|_n|\end{cases}$,
			\item $2^x_n=\begin{cases}2^x_0=x \\2^x_{n+1}=2^{2^x_n}\end{cases}$.
		\end{enumerate}
	\end{definition}
	\begin{definition}
		For every $k$, define k'th {\bf Exponential Time Hierarchy} ( $\E\H_k$)  as follows:
		\begin{itemize}
			\item For every $L\subseteq \N$, $L$ is in $\E_k$ iff there exists a $\Delta^b_1$ formula $\phi(x)$ in $\s$ such that $\all n(n\in L \lr \phi(2^n_k))$,
			\item For every $L\subseteq \N$, $L$ is in $\Sigma^{\E_k}_i$ for some $i>0$ iff there exists a $\Sigma^b_i$ formula $\phi(x)$ such that $\all n(n\in L \lr \phi(2^n_k))$,
			\item For every $L\subseteq \N$, $L$ is in $\Pi^{\E_k}_i$ for some $i>0$ iff there exists a $\Pi^b_i$ formula $\phi(x)$ such that $\all n(n\in L \lr \phi(2^n_k))$.
		\end{itemize}
	\end{definition}
	Note that we do not have a exponentiation function symbol in ${\cal L}_{BA}$, therefore by formula  $\all n\phi(2^{f(n)}_k)$ for some poly time function $f$ and some fix $k$, we mean $\all m,n(\psi_{f,k}(m,n)\to \phi(m))$ in which $\psi_{f,k}(m,n)$ is a $\Delta^b_1$ formula in $\s$ that is true iff $m=2^{f(n)}_k$.
	\begin{proposition}\label{t1}
		The following statements are true:
		\begin{enumerate}
			\item 
			If for every $T\in\mathcal{T}$, there exists $S\in\mathcal{T}$ such that the $T$-proofs of $\Sigma^b_1\rf_S(\bar{n})$ are not polynomially bounded in $n$, then $\NE\not = \Sigma^\E_2$.
			\item If for every $T\in\mathcal{T}$, there exists $S\in\mathcal{T}$ such that the $T$-proofs of $ \Sigma^b_1\rf_S(\bar{n})$ can not be constructed in polynomial time, then $\E\not=\Sigma^\E_2$.
		\end{enumerate}
	\end{proposition}
	\begin{proof}
		Here we prove the statement (1). The statement (2) has a similar proof. Let $\NE=\Sigma^\E_2$. This implies that $\NE=\Pi^\E_2$, because $\Co\NE\subseteq \Sigma^\E_2$. Define the following languages: 
		\begin{enumerate}
			\item $L_\NE = \{n=\left<e,x,m\right >\in\N :\N\models \mu_1(e,x,2^{2^{|m|}})\}\in \NE$.
			\item $L_{\Pi^\E_2}=\{n=\left<e,x,m\right >\in\N :\N\models \neg\mu_2(e,x,2^{2^{|m|}})\}\in\Pi^\E_2$.
		\end{enumerate}
	Note that the above languages are hard for their respective complexity class under linear time reductions.	By definition there exist the following predicates:
		\begin{enumerate}
			\item There exists a $\Pi^b_2$ predicate $\UU$ such that $\N\models \all n(\UU(2^n)\lr n\in L_{\Pi^\E_2})$,
			\item There exists a $\NP$ predicate $\UN$ such that $\N\models \all n(\UN(2^n)\lr n\in L_{\NE})$.
		\end{enumerate}
		Note that $\NE=\Pi^\E_2$ implies that there exists a linear time function $f$ such that $$\N\models \all n(\UU(2^n)\lr\UN(2^{f(n)})).$$
		Let $T\in\mathcal{T}$ be a theory with the following properties:
		\begin{enumerate}
			\item $T\vdash \UN(2^n)\text{ is }\NE$-hard with respect to linear time reductions,
			\item $T\vdash \UU(2^n)\text{ is }\Pi^\E_2$-hard with respect to linear time reductions,
			\item $T\vdash  \all n(\UU(2^n)\lr\UN(2^{f(n)}))$
		\end{enumerate}
		Let $T'$ be in $\mathcal{T}$. This implies $\Sigma^b_1\rf_{T'}(x)\in\Pi^\E_2$, so by the mentioned properties of $T$ there exists a linear time function $g$ such that $T\vdash  \all n\big(\Sigma^b_1\rf_{T'}(n)\lr\UN(2^{f(g(n))})\big)$. Because $\UN(x)$ is $\Sigma^b_1$ and also $\s\subseteq T$, there exists a polynomial $r(x)$ such that $$T\vdash \all x\big(\UN(x)\to\ex y\big(|y|\leq r(|x|) \land Pr_T\big(y,\ulcorner\UN(\dot{x})\urcorner\big)\big)\big).$$
		This implies $$T\vdash \all x\big(\UN(2^{f(g(x))})\to\ex y\big(|y|\leq r(f(g(x))+1) \land Pr_T\big(y,\ulcorner\UN(2^{f(g(\dot{x}))})\urcorner\big)\big)\big).$$
		Note that $\N\models \all n\UN(2^{f(g(n))})$, so for every $n\in\N$, $T\sststile{}{r(f(g(n))+1)} \UN(2^{f(g(\bar{n}))})$, hence there exists a polynomial $p(x)$ such that for every $n\in \N$, $T\sststile{}{p(n)} \Sigma^b_1\rf_{T'}(\bar{n})$.
	\end{proof}
	In the next theorem, we will investigate how much optimality we can get by assuming the equality of nondeterministic and (co-non)deterministic computation for fat classes in sense of \cite{bg98}, such as ${\EXP}$ and ${\E_k}$ for $k>2$. To state the theorem, we need some definitions. Let $2^{o(n)}$ and $2^{(\log n)^{O(1)}}$ be sub- exponential (subExp) and quasi polynomial (Qp) respectively. The concept of simulations and reductions can be defined in terms of other time classes like sub-exponential or quasi-polynomial time instead of polynomial time and the relations in figure \ref{fig} remain true, hence it is natural to ask whether these new conjectures are true or not. An oracle is constructed in \cite{gssz04} that $\disj\NP$ pairs do not have complete problems with respect to the poly time reductions. It is not hard to modify that construction to make an oracle in which $\disj\NP$ pairs do not have complete problem, with respect to sub-exponential time reductions, hence conjectures weaker than it are true with respect to that oracle. For the other branch, we will construct an oracle that $\disj\Co\NP$ pairs do not have a complete problem, with respect to poly time reductions and it is easy to modify the construction in such a way that $\disj\Co\NP$ pairs do not have a complete problem, with respect to sub-exponential time reductions. Hence, the oracles provide evidence that these new conjectures are true. 
	\begin{theorem}\label{t2}
		The following statements are true:
		\begin{enumerate}
			\item If there is no nonuniform subExp-optimal proof system for $\taut$, then for every $k$, $\NE_k\not=\Co\NE_k$.
			\item If there is no subExp-optimal proof system for $\taut$, then for every $k$, $\E_k\not=\NE_k$.
			\item If there is no nonuniform Qp-optimal proof system for $\taut$, then $\NEXP\not=\Co\NEXP$.
			\item If there is no Qp-optimal proof system for $\taut$, then $\EXP\not=\NEXP$.
		\end{enumerate}
	\end{theorem}
	\begin{proof}
		Here we only prove the statement (1). The proofs of the other statements are similar. Let $\NE_k=\Co\NE_k$ for some $k>0$. Define the following complete languages: 
		\begin{enumerate}
			\item $L_{\NE_k} = \{n=\left<e,x,m\right >\in\N :\N\models \mu_1(e,x,2^{|m|}_{k+1})\}\in \NE_k$.
			\item $L_{\Co\NE_k} = \{n=\left<e,x,m\right >\in\N :\N\models \neg\mu_1(e,x,2^{|m|}_{k+1})\}\in \Co\NE_k$.
		\end{enumerate}
	Note that the above languages are hard for their respective complexity class under linear time reductions. By definition there exist the following predicates:
		\begin{enumerate}
			\item There exists a $\NP$ predicate $\UN$ such that $\N\models \all n(\UN(2^n_k)\lr n\in L_{\NE_k})$,
			\item There exists a $\Co\NP$ predicate $\UC$ such that $\N\models \all n(\UC(2^n_k)\lr n\in L_{\Co\NE_k})$.
		\end{enumerate}
		Note that $\NE_k=\Co\NE_k$ implies that there exists a linear time  function $f$ such that $$\N\models \all n(\UC(2^n_k)\lr\UN(2^{f(n)}_k)).$$
		Let $T\in\mathcal{T}$ be a theory with the following properties:
		\begin{enumerate}
			\item $T\vdash \UN(2^n_k)\text{ is }\NE_k$-hard with respect to linear time reductions,
			\item $T\vdash \UC(2^n_k)\text{ is }\Co\NE_k$-hard with respect to linear time functions,
			\item $T\vdash  \all n(\UC(2^n_k)\lr\UN(2^{f(n)}_k))$
		\end{enumerate}
		Let $T'$ be in $\mathcal{T}$. For every $i$, define $\Con^i_{T'}(x):=\all y(|y|_i\leq x\to \neg Pr_{T'}(y,\ulcorner \bot \urcorner)$, hence $\Con^k_{T'}(x)\in\Co\NE_k$. So by the mentioned properties of $T$ there exists a linear time function $g$ such that $T\vdash  \all n\big(\Con^k_{T'}(n)\lr\UN(2^{f(g(n))}_k)\big)$. Because $\UN(x)$ is $\Sigma^b_1$ and also $\s\subseteq T$, there exists a polynomial $r(x)$ such that $$T\vdash \all x\big(\UN(x)\to\ex y\big(|y|\leq r(|x|) \land Pr_T\big(y,\ulcorner\UN(\dot{x})\urcorner\big)\big)\big).$$
		This implies $$T\vdash \all x\big(\UN(2^{f(g(x))}_k)\to\ex y\big(|y|\leq r(2^{f(g(x))}_{k-1}+1) \land Pr_T\big(y,\ulcorner\UN(2^{f(g(\dot{x}))}_k)\urcorner\big)\big)\big).$$
		Note that $\N\models \all n\UN(2^{f(g(n))}_k)$, so for every $n\in\N$, $T\sststile{}{r(2^{f(g(n))}_{k-1}+1)} \UN(2^{f(g(\bar{n}))}_k)$, hence there exists a polynomial $p(x)$ such that for every $n\in \N$, $T\sststile{}{p(2^{f(g(n))}_{k-1})} \Con^k_{T'}(\bar{n})$, hence $T\sststile{}{p(2^{f(g(|n|_{k-1}))}_{k-1})} \Con^k_{T'}(|\bar{n}|_{k-1})$, so there exists a polynomial $q(x)$ such that for every  $n\in \N$, $T\sststile{}{q(2^{f(g(|n|_{k-1}))}_{k-1})} \Con^1_{T'}(\bar{n})$. Note that there exists $0<\epsilon<1$ such that $q(2^{f(g(|n|_{k-1}))}_{k-1})=O(2^{n^\epsilon})$. By the fact that proof of theorem \ref{t2.2} is adoptable in case of quasi polynomial and sub-exponential, the proof is completed.
	\end{proof}
	Note that similar theorems can be proved for $\rfnN$ and $\rfn$. The main problem in proof os theorem \ref{t2} that does not permit us to prove that nonexistence of nonuniform p-optimal proof systems implies separation of $\NE_k$ and $\Co\NE_k$ for $k>1$, is that these classes are not closed under reductions, but we can separate these classes if we strengthen our assumption like the following theorem.
	\begin{theorem}\label{t4.3}
		َLet $k>0$, then at least one of the following statement is true:
		\begin{enumerate}
			\item There is no recursive function $F(x)$ such that $$\N\models \all e,x(\neg\mu_1(e,2^x_k,2^{(2^x_{k-1}+1)^{e}}) \lr \mu_1(F(e),2^x_k,2^{(2^{x}_{k-1}+1)^{F(e)}})),$$
			\item There is no nonuniform p-optimal proof system for $\taut$.
		\end{enumerate}
		Also, a similar statement is true for p-optimality and equality of $\E_k$ and $\NE_k$.
	\end{theorem}
	\begin{proof}
		Let $(1)$ be false. This implies that we can find a theory $T\in\mathcal{T}$ such that it effectively proves $\NE_k=\Co\NE_k$ and because $T$ is $\Sigma_1$-complete, $T$ can prove $\Con^k_{T'}(x)$ for some $T'\in\mathcal{T}$ is equivalent to $\phi(2^x_k)$ for some $\phi\in\Sigma^b_1$. The rest of the proof is like the proof of theorem \ref{t2}.
	\end{proof}
	
	Theorem \ref{t4.3} has interesting corollaries.
	\begin{corollary}
		The following statements are true:
		\begin{enumerate}
			\item If there is no nonuniform p-optimal proof system for $\taut$, then for every $T\in{\cal T}$ and for every $k>0$, there is a $\Pi^b_1$ formula $\phi$ such that for every $\Sigma^b_1$ formula $\psi$, $T\not\vdash \all n(\phi(2^n_k)\lr \psi(2^n_k))$.
			\item If there is no p-optimal proof system for $\taut$, then for every $T\in{\cal T}$ and for every $k>0$, there is a $\Delta^b_1$ formula $\phi$ in $\s$ such that for every $\Sigma^b_1$ formula $\psi$, $T\not\vdash \all n(\phi(2^n_k)\lr \psi(2^n_k))$.
		\end{enumerate}
	\end{corollary}
	\begin{proof}
		The following argument is working for both cases. Suppose $k$ is fixed. If there is a $T\in{\cal T}$ such that for every $\Pi^b_1$ formula $\phi$, there exists a $\Sigma^b_1$ formula $\psi$ such that $T\vdash \all n(\phi(2^n_k)\lr \psi(2^n_k))$, then the following algorithm defines a recursive function, by giving an input $e$, enumerate all $T$-proofs and for every proof check whether it is a $T$-proof of $\all n(\neg\mu_1(e,2^x_k,2^{(2^x_{k-1}+1)^{e}})\lr \phi(2^n_k))$ for some $\Sigma^b_1$ formula $\phi$. Note that this enumeration and checking process is recursive because the axioms of $T$ are poly time decidable. Also, note that by assumption this algorithm always finds such a $\psi$, hence we can find its code and output it. Thus, according to theorem \ref{t4.3} there is a nonuniform p-optimal proof system.
	\end{proof}
	The next corollary shows that theorem \ref{t4.3} implies conditional independence for strong intuitionistic theories.
	\begin{corollary}
		Let $T$ be an intuitionistic theory such that any arithmetical theorem of $T$ is recursively realizable, then:
		\begin{enumerate}
			\item 
			If there is no nonuniform p-optimal proof system for $\taut$, then for every $k>0$, $T\nvdash\NE_k=\Co\NE_k$,
			\item If there is no p-optimal proof system for $\taut$, then for every $k>0$, $T\nvdash\E_k=\NE_k$.
		\end{enumerate}
	\end{corollary}
	\begin{proof}
		Note that by $\NE_k=\Co\NE_k$ we mean the natural formalization $$\all e\ex e'\all x(\neg\mu_1(e,2^x_k,2^{(2^x_{k-1}+1)^{e}}) \lr \mu_1(e',2^x_k,2^{(2^{x}_{k-1}+1)^{e'}})).$$ If $T\vdash \NE_k=\Co\NE_k$ for some $k>0$, it actually give us a recursive function $F(x)$ such that $\N\models \all e,x(\neg\mu_1(e,2^x_k,2^{(2^x_{k-1}+1)^{e}}) \lr \mu_1(F(e),2^x_k,2^{(2^{x}_{k-1}+1)^{F(e)}}))$ by recursive realizability, hence by theorem \ref{t4.3} it implies the existence of a nonuniform p-optimal proof system for $\taut$. The proof of the second statement is similar.
	\end{proof}
	Note that arithmetical theorems of strong intuitionistic theories like ${\sf HA}$ (Heyting Arithmetic), ${\sf CZF}$ (Constructive Zermelo-Fraenkel) and ${\sf IZF}$ (Intuitionistic Zermelo-Fraenkel) are recursively realizable. For more information about the soundness of these theories with respect to the recursive realizability see \cite{tv88} and \cite{rat06}.
	\section{Relativized worlds}
	In this section, we will construct two oracles which imply several separations between conjectures of the two branches in figure \ref{fig}. Our constructions are based on the usual definition of forcing in arithmetic.
	\begin{definition}
		A nonempty set $\cal P$ of functions from natural numbers to $\{0,1\}$ (for every $p\in{\cal P}$, $\D(p)\subseteq \N$ and $\R(p)\subseteq\{0,1\}$ ) is a forcing notion iff for every $p\in {\cal P}$, there exists a $q\in{\cal P}$ such that $p\subsetneq q$. We call members of a forcing notion a condition.
	\end{definition}
	Let $\alpha$ be a new unary relation symbol. For every $p\in{\cal P}$ and every ${\cal L}_{BA}(\alpha)$ sentence $\phi$ we will define $p\Vdash \phi$ by induction on the complexity of $\phi$ as follows:
	\begin{enumerate}
		\item $p\not\Vdash \bot$,
		\item $p\Vdash s=t$, iff $\N\models s=t$,
		\item $p\Vdash \alpha(t)$ for some closed term $t$, iff $p(t)=1$,
		\item $p\Vdash \neg\psi$, iff for every $q\in{\cal P}$ such that $p\subseteq q$, $q\not\Vdash \psi$,
		\item $p\Vdash \psi \lor \eta$, iff $p\Vdash \psi$ or $p\Vdash \eta$,
		\item $p\Vdash \psi \land \eta$, iff $p\Vdash \neg(\neg\psi\lor \neg\eta)$,
		\item $p\Vdash \ex x\psi(x)$, iff there exists $n\in\N$ such that $p\Vdash \psi(n)$,
		\item $p\Vdash \all x\psi(x)$, iff $p\Vdash\neg \ex x\neg\psi(x)$.
	\end{enumerate}

	Our constructions can be done in the usual density argument in forcing, but we present our arguments in the constructive extension fashion, because it is more readable. For the next theorem we use the forcing notion ${\cal P}=\{p:p\text{ is a finite function from }\N\text{ to }\{0,1\}\}$. In the rest of the paper we use notation $[n]=\{0,1,...,n\}$. Also, by $t_A(n)$ for some computational machine $A$ ($\sf FP$ functions, $\Sigma^b_i$ relations, etc) we mean the time complexity of $A$ on inputs with length of $n$.
	\begin{theorem}\label{t5.1}
		There exists an oracle $\V$ such that $\disconp^\V$ is true, but $\E^\V=\NE^\V$.
	\end{theorem}
	\begin{proof}
		Let $\{(\phi_i,\psi_i,R_i)\}_{\i\in\}_{i\in\N}}$ be an enumeration of $\Pi^b_1(\alpha) \times \Pi^b_1(\alpha) \times\F\P^\alpha$. We want to construct a sequence $p_0\subseteq p_1\subseteq p_2\subseteq ...$  of ${\cal P}$ such that $\V=\bigcup_i p^{-1}_i(1)$ and $\disconp^\V$ is true, but $\E^\V=\NE^\V$ if $\alpha$ is interpreted by $\V$.
		
		For every $i$ define the following $\Pi^b_1(\alpha)$ sets:
		\begin{enumerate}
			\item $L^1_i=\{w:\all |y|=|w|(2\left <i,1,w,y\right >\in \alpha)\}$,
			\item $L^2_i=\{w:\all |y|=|w|(2\left <i,2,w,y\right >\in \alpha)\}$.
		\end{enumerate}
		
		For every $i$, let $r_i$ be the first index of occurrence of $(\phi_i,\psi_i)$ in the enumeration $\{(\phi_i,\psi_i,R_i)\}_{i\in\N}$. We want to construct $\V$ such that for every $i$, either $(\phi_i,\psi_i)$ is not disjoint or $(L^1_{r_i},L^2_{r_i})$ is disjoint and it is not reducible to $(\phi_i,\psi_i)$ by $R_i$. Let $L_\NE$ be the relativized version of the $\NE$-complete problem defined in theorem \ref{t1} and $\UN(x)$ be a $\Sigma^b_1(\alpha)$ predicate such that $$(\N,A)\models \all n(n\in L\lr\UN(2^n)).$$
		 for every $A$. Let $t_{\UN}(n)\leq n^c+c$ for some $c>0$. We want to code membership of $L$ in $\V$ to make sure that $\E^\V=\NE^\V$. We use the following coding for this matter: $$(\N,\V)\models\all n(n\in L \lr 2^{(n+1)^c+c}+1\in \alpha).$$
		Note that $\UN(2^n)$ can not query $2^{(n+1)^c+c}+1$. Suppose we construct $p_{i-1}:\D(p_{i-1})\to\{0,1\}$.
		Let $m$ be big enough (we compute how big $m$ should be). Suppose $\max(t_{\phi_i}(m),t_{\psi_i}(m),t_{R_i}(m))\leq m^d+d$. Define $p_{i-1}\subseteq q$ as follow:
		\begin{enumerate}
			\item $\D(q)\subseteq [2^{m^d+d}]$,
			\item $\{2\left <r_i,v,x,y\right >: |x|=|y|=m, v\in\{1,2\}\}\cap \D(q) =\varnothing$,
			\item $(\D(q)\setminus \D(p_{i-1}))\cap\{2^{(n+1)^c+c}+1:n\in\mathbb{N}\}=\varnothing$,
			\item $\{2\left <a,v,x,y\right >: a,x,y\in\mathbb{N},v\in\{1,2\},|x|=|y|,|x|\not=m\}\setminus \D(p_{i-1})\subseteq q^{-1}(0)$
		\end{enumerate}
		Now we want to extend $q$ to make sure the coding requirement. Let $u_0=q$. For each $j>0$ such that $2^{(j+1)^c+c}+1<2^{m^d+d}$ we construct $u_j$ by the following rules:
		\begin{enumerate}
			\item If $2^{(j+1)^c+c}+1\in \D(u_{j-1})$, then put $u_j=u_{j-1}$,
			\item otherwise, 
			\begin{enumerate}
				\item if $u_{j-1}\Vdash \neg \UN(2^j)$, put $u_j=u_{j-1}\cup\{(2^{(j+1)^c+c}+1,0)\}$,
				\item  otherwise, extend $u_{j-1}$ to $u_j$ such that:
				\begin{itemize}
					\item $u_j\Vdash \UN(2^j)$,
					\item $2^{(j+1)^c+c}+1\in u_j^{-1}(1)$,
					\item $|u_j\setminus u_{j-1}|\leq (j+1)^c+c+1$, we can force this condition because only we need to know the queries of $\UN(2^j)$ in its accepting path.
				\end{itemize}
				
			\end{enumerate}
		\end{enumerate}
		Let $q'$ be unions of $u_j$ for $2^{(j+1)^c+c}+1<2^{m^d+d}$. For each $x$ such that $|x|=m$, define $S_x=\{2\left <r_i,v,x,y\right >: |y|=m, v\in\{1,2\}\}$. Let $k=|\{j\in\mathbb{N}:2^{(j+1)^c+c}+1<2^{m^d+d}\}|$, therefore we have:
		$$|q'\setminus q|\leq \sum_{j=0}^{k-1}(j+1)^c+c+1\leq k(k^c+c+1).$$
		Because $k\leq (m^d+d-c)^{\frac{1}{c}}$, we have $|q'\setminus q|\leq  (m^d+d-c)^{\frac{1}{c}}(m^d+d+1)$. If $m$ is big enough, then $\max\{(m^d+d-c)^{\frac{1}{c}}(m^d+d+1),3(m^d+d)\}<2^m$ which means there exists $z$ with length of $m$ such that $S_z\cap \D(q')=\varnothing$. Note that by our construction $q'\not\Vdash \ex x(x\in L^1_{r_i}\land x\in L^2_{r_i})$. Now we have enough rooms to extend $q'$ in such a way that either $(\phi_i,\psi_i)$ is not disjoint or $(L^1_{r_i},L^2_{r_i})$ is not reducible to $(\phi_i,\psi_i)$ by $R_i$. We compute $R_i(z)$ and answer new oracle questions by the following rule:
		\begin{enumerate}
			\item For every oracle question $y$, if $y\in S_z$, then accept $y$ and put $y$ in $\mathcal{A}$,
			\item if $(y,1)\in q'$ accept $y$,
			\item otherwise, reject $y$.
		\end{enumerate}
		Let $R_i(z)=z^*$. Let $\PP^*\subseteq \PP$ such that for every $u\in{\cal P}^*$, the following properties are true:
		\begin{enumerate}
			\item $\D(u)\subseteq [2^{m^d+d}]$,
			\item $u|_{\D(q')}=q'$,
			\item $\mathcal{A}\subseteq u^{-1}(1)$,
			\item $u^{-1}(0)\cap S_z=\varnothing$,
			\item $|\D(u)\cap S_z|\leq 2(m^d+d)$.
		\end{enumerate}
		Now there are two cases that can occur:
		\begin{enumerate}
			\item 
			If for every $u\in \PP^*$, $u\nVdash \neg\phi_i(z^*)$ and also $u\nVdash \neg\psi_i(z^*)$, then define $p':[2^{m^d+d}]\to \{0,1\}$ by the following definition:$$p'(c)=\begin{cases}q'(c) & c\in \D(q')\\1 & c\in S_z\\0 & \text{o.w.}\end{cases}$$
			Note that $p'\nVdash \neg\phi_i(z^*)$ and also $p'\nVdash \neg\psi_i(z^*)$, because if for example $p'\Vdash \neg\phi_i(z^*)$, then there exists a subset $F\subseteq [2^{m^d+d}]$ such that $p'|_{F}\in\PP^*$ and $p'|_{F}\Vdash \neg\phi_(z^*)$ which contradicts our assumption, hence $p'\nVdash \neg\phi_i(z^*)$ and also $p'\nVdash \neg\psi_i(z^*)$, but this implies $p'\Vdash \phi_i(z^*)\land\psi_i(z^*)$, because $p'$ has answers for the oracle questions for all of the numbers with length of less than $m^d+d+1$. This means that $\phi_i$ and $\psi_i$ are not disjoint relative to our construction and we define $p_i$ as $p'$.
			\item Otherwise, without loss of generality we can assume that there exists a $u\in\PP^*$ such that $u\Vdash \neg\phi_i(z^*)$. Let $S=\{2\left <r_i,1,z,y\right >: |y|=m\}$ and define $p_i$ as a condition by the following properties:
			\begin{enumerate}
				\item $\D(p_i)=[2^{m^d+d}]$,
				\item $u\subseteq p_i$,
				\item $S\subseteq p^{-1}_i(1)$,
				\item $[2^{m^d+d}]\setminus(\D(u)\cup S)\subseteq p^{-1}_i(0)$.
			\end{enumerate}
			Therefore, we have the following facts:
			\begin{enumerate}
				\item $p_i\Vdash \neg\phi_i(z^*)$,
				\item $p_i\Vdash z\in L^1_{r_i}$.
			\end{enumerate}
			This implies that $(L^1_{r_i},L^2_{r_i})$ is not reducible to $(\phi_i,\psi_i)$ by $R_i$, relative to our construction.
		\end{enumerate}
		By explanations of the above cases our oracle construction is completed.
	\end{proof}
	In the rest of the paper we want to construct an oracle $\cal W$ such that $\tfnp^{\cal W}={\sf FP}^{\cal W}$, but there is no nonuniform p-optimal proof system for $\taut^{\cal W}$. We will use the Kolmogorov generic construction idea that is defined in \cite{bfkrv10}. Here we borrow definitions and notations from \cite{bfkrv10}. Note that because we explained how to code binary strings in natural numbers and vice versa, we use both natural numbers and strings in the rest of the paper without loss of generality.
	\begin{definition}
		For every partial computable function $F(x,y)$ and every $x,y\in\{0,1\}^*$, the Kolmogorov complexity of $x$ conditional to $y$ with respect to $F$, which will be denoted as $C_F(x|y)$, is defined as follows:$$C_F(x|y)=\min\{|e|:e\in\{0,1\}, F(e,y)=x\}$$
	\end{definition}
	We will say that $C_F(x|y)$ for some partial computable function $F(x,y)$ is a universal method iff for every partial computable $G(x,y)$, there exists a constant $k$ such that $$\all x,y\in\{0,1\}^*(C_F(x|y)\leq C_G(x|y)+k).$$ 
	According to the Solomonoff-Kolmogorov theorem there exists a universal method. We will show it by $C(x|y)$. Also, we define the unconditional Kolmogorov complexity of $x$ with $C(x)=C(x|\lambda)$ in which $\lambda$ is the empty string.
	Here we list some properties of Kolmogorov complexity that are stated in \cite{bfkrv10}.
	\begin{enumerate}
		\item For all $x$ and $y$, $C(x|y)\leq C(x)+O(1)$.
		\item There exists a constant $k$ such that for all $x$, $C(x)\leq |x|+k$.
		\item  For all $n$ and $m$, there is an $n$ bit string $x$ such that $C(x)\geq n-m$. In particular, for every $n$ there is an $n$ bit string $x$ such that $C(x)\geq n$. Such strings are called incompressible. 
		\item For every computable function $f(x_1,...,x_n)$, $$C(f(x_1,...,x_n))\leq 2|x_1|+2|x_2|+...+2|x_{n-1}|+|x_n|+O(1).$$
	\end{enumerate} 
	For every $n>0$ fix a $n2^n$ bit string $Z_n$ such that $C(Z_n)\geq n2^n$. Divide $Z_n$ into $2^n$ string $z^n_1$ to $z^n_{2^n}$, each of length $n$. Define ${\cal K}=\{\llcorner\left<i,z^j_i\right >\lrcorner:\ex k\in\N(j=2^1_k), i\in\{0,1\}^j\}$. We define the forcing notion ${\cal P}_K=\{p: p\text{ is a function from }{\cal K}\text{ to }\{0,1\}, {\cal K} \setminus \D(p) \text{ is infinite} \}$.
	\begin{theorem}\label{t5.2}
		There exists an oracle $\cal W$ such that there is no nonuniform p-optimal proof system for $\taut^{\cal W}$, but $\tfnp^{\cal W}={\sf FP}^{\cal W}$.
	\end{theorem}
	\begin{proof}
		Following the argument in \cite{bfkrv10}, we construct an oracle $\W$ such that there is no nonuniform p-optimal proof system for $\taut^\W$, but $\tfnp^\W = {\sf FP}^\W$, assuming $\sf FP = FPSPACE$.
		As we will see, the oracle construction still works if we first relativize things with a $\sf PSPACE$-complete set $H$ and then construct $\W$ with the desired properties. Note that relativizing to
		$H$ implies ${\sf FP}^H = {\sf FPSPACE}^H$ and hence we are free from the assumption $\sf FP = FPSPACE$.
		Also, note that relativizing first to $H$ and then relativizing to $\W$ is equivalent to relativizing with $H \oplus \W$ in which $A \oplus B =\{ 2n : n\in A\}\cup\{ 2n + 1 : n \in B\}$. Let $\{f_i(x)\}_{i\in\N}$
		and $\{(r_i, \phi_i(x, y))\}_{i\in\N}$ be enumerations of ${\sf FP}(\alpha)$ functions and $\N \times\Delta^b_1(\alpha)$ in which $\phi_i(x, y)$
		defines a poly time relation with access to $\alpha$. In the rest of the proof we construct a
		sequence $p_0 \subseteq p_1 \subseteq...$ of ${\cal P}_K$ such that $\W = \bigcup_i p^{-1}_i(1)$ and there is no nonuniform p-optimal proof system for $\taut^\W$, but $\tfnp^\W = {\sf FP}^\W$ if $\alpha$ is interpreted by $\W$. For every $i, k \in \N$ define $\theta_{i,k}$ be the Paris-Wilkie translation of $\Pi^b_1(\alpha)$ sentence
		$\all x(|x| =\bar{3}\bar{n} +\bar{3}\to \neg\alpha(x))$ in which $n = 2^1_{\left<i,k\right>}$. For every $i, j \in \N$ define $S^i_j = \{\theta_{i,k} : k \geq j \}$
		and $B^i_j = \{x : x \in {\cal K}, |x| = 3(2^1_{\left<i,j\right>} + 1)\}$. Suppose we construct $p_{i-1} : \D(p_{i-1})\to \{0, 1\}$. We
		extend $p_{i-1}$ to $p_i$ as follows:
		\begin{enumerate}
			\item  If $i = 2a$, then we want to make sure that $f_a$ will not be a proof system or $f_a$ will
			not have short proofs for members of the set $S^a_{c_a}$ for some $c_a$ relative to $\W$. Let
			$t_{f_a}(n) \leq n^d + d$. Choose $c_a$ such that $\D(p_{i-1})\cap\left(\bigcup_{c_a\leq j}B^a_j\right)=\varnothing$ and also for every $n \geq c_a$,
			$4nd^{d\log_2 4n} + d < 2^n$. Now, there are two cases that can happen:
			\begin{enumerate}
				\item There is a $p_{i-1} \subseteq q\in {\cal P}_K$, some $\theta\in S^a_{c_a}$ and $\pi\in \N$ such that
				$$q \Vdash |\pi| \leq |\theta|^{d\log_2 |\theta|} + d \land f_a(\pi) = \theta.$$
				This implies that there is a $p_{i-1} \subseteq q' \in {\cal P}_K$ such that $|\D(q') \setminus \D(p_{i-1})| \leq |\theta|^{d \log_2 |\theta|} + d$
				and $q'\Vdash |\pi|\leq |\theta|^{d \log_2 |\theta|} + d\land f_a(\pi) = \theta$, because $f_a$ only needs at most $|\theta|^{d \log_2 |\theta|} +
				d$ query answers from $\W$ on input $\pi$. Let $\theta$ be $\theta_{a,k}$ for some $k$. This means
				$|\theta|^{d \log_2 |\theta|} + d < |B^a_k| = 2^m$ in which $m = 2^1_{ \left<a,k\right>}$, hence there is a $z \in B^a_k \setminus \D(q')$. Define
				$p_i := q'\cup \{ (z, 1)\}$. This implies that $f_a$ relative to $\W$ will not be a proof system for $\taut^\W$,
				because it proves $\theta_{a,k}$, but $\theta_{a,k}$ is not a tautology relative to $\W$,
				\item otherwise, we define $p_i := p_{i-1}\cup\{(x,0):\ex k\in\N(k\geq c_a \land x\in B^a_k)\}$. Note that
				in this case, for every $\theta\in S^a_{c_a}$, there is no $|\theta|^{d \log_2 |\theta|} + d$ length proof of $\theta$ in $f_a$
				relative to $\W$.
			\end{enumerate}
			So by construction of $p_i$ we make sure that $f_a$ is not a proof system or $f_a$ is not a
			nonuniform p-optimal proof system for $\taut^\W$, because $S^a_{c_a}$ is poly time decidable.
			\item  If $i = 2a + 1$, then we want to make sure that $(n^{r_a} + r_a, \phi_a(x, y))$ will not define a
			$\tfnp$ problem relative to $\W$ or it can be computed by some function in ${\sf FP}^\W$. The
			construction in this case is very easy. If there is a $p_{i-1} \subseteq q \in {\cal P}_K$ such that
			$q\Vdash \ex x\all y(|y|\leq |x|^{r_a}+r_a\to\neg\phi_a(x,y))$, then there is some $p_{i-1} \subseteq q'\in {\cal P}_K$ such that
			$|\D(q')\setminus \D(p_{i-1})|$ is finite and $q'\Vdash \ex x\all y(|y|\leq |x|^{r_a}+r_a\to\neg\phi_a(x,y))$. In this case we define
			$p_i := q'$, otherwise if there is no such extension, then we define $p_i := p_{i-1}$. 
		\end{enumerate}
		Suppose
		$(n^{r_a} + r_a, \phi_a(x, y))$ defines a $\tfnp$ problem relative to $\W$. Now we want to show there
		is a function $f\in {\sf FP}^\W$ such that it solves $(n^{r_a} + r_a; \phi_a(x, y))$. Let $t_{\phi_a}(x,y)\leq (|x|+ |y|)^b + b$,
		then on input $u$ with solution $v$, $\phi_a(u, v)$ asks at most $(|u| + |u|^{r_a} + r_a)^b + b$ questions
		from $\W$. Choose $e$ such that for all $n$, $(n+ n^{r_a} + r_a)^b + b \leq n^e + e$. The function $f$ works as follows on input $x$:
		
		Let $m = 2^1_k$ be the biggest tower of two such that $m \leq 4|x|^{2e}$. Note
		that to compute a solution of this problem we only need to know the oracle
		answers for members $\bigcup_{i\leq m} Y_i$. First, $f$ asks the value of $\W$ for every member of
		$\bigcup_{i\leq \log_2 m} Y_i$ and puts the answers in $G$. Then it proceeds as the following procedure by
		starting with $Q_1 = \varnothing$:
		In the $i$'th iteration, using the power of $\sf PSPACE$ (we assumed that $\sf FP = FPSPACE$)
		find the least $|v_i| \leq |x|^{r_a} + r_a$ such that $\phi_a(x, v_i)$ is true relative to $G\cup Q_i$. If
		$\phi_a(x, v_i)$ is true relative to $\W$, then halt and output $v_i$, otherwise there is a
		$u_i \in (\W \cap Y_m) \setminus Q_i$ such that it is the first number in which it is queried in
		computation of $\phi_a(x, v_i)$ relative to the $\W$ such that $u_i \in \W$, but $u\not\in Q_i$. Define
		$Q_{i+1} = Q_i\cup\{u_i\}$ and repeat this procedure.
		
		First, note that in every iteration, this procedure indeed finds a $v$ such that
		relative to $G\cup Q_i$, $\phi_a(x, v)$ holds, because in that case we can find a condition $p_i \subset    q \in {\cal P}_K$ such that $G\cup Q_i\subseteq q^{-1}(1)$ and hence $q$ forces that $(n^{r_a}+r_a, \phi_a(x, y))$ is not a
		$\tfnp$ problem (note that if $Y_m\cap \W = \varnothing$, then we should find the solution
		of the problem relative to $\W$ in the first iteration, hence the construction of the
		previous conditions which make sure some proof systems are not nonuniformly p-optimal will not cause a problem in finding such a $q$). After some iterations
		$f$ will find a solution of this $\tfnp$ problem relative to $\W$. If we prove that the
		number of iterations are polynomial in $|x|$, then we are done. Suppose after $l$'th
		iteration we find the solution. This means that $|Q_l| = l - 1$. Let $l' = l - 1$.
		Note that for every $j < l$, $u_j$ can be described by the code of poly time relation $\phi_a(x, y)$, $x$, $G\cup Q_j$
		and an $e\log_2 |x|$ bit string which it shows the order number of $u_j$ among
		the queries of $\phi_a(x, v_j)$, hence $Q_l$ can be described by a string of length $l'(e\log_2 |x|)+
		O(m\log_2 m) + 2|x| + O(1)$ (note that $G$ has at most $m + \log_2 m + \log_2 \log_2 m + ...$
		of strings of length at most $\log_2 m$, hence it can be described by a string of length $O(m\log_2 m)$ bits). Let $p$ be the concatenation of all $y$'s from
		$\llcorner \left<i,y\right>\lrcorner\in Y_m\setminus Q_l$ according to the order on $i$'s, hence $|p| = m(2^m − l')$. Note that
		$Z_m$
		can be described using $p$ by inserting the second component of members of
		$Q_l$ in places that the first component refer to, hence by the fact that $C(Q_l) \leq
		l'(e\log_2 |x|) + O(m\log_2 m) + 2|x| + O(1)$, we have:
		$$m2^m \leq C(Z_m) \leq m(2^m − l') + 2l'(e\log_2 |x|) + O(m\log_2 m) + 4|x| + O(1).$$
		This implies $l'(m−2e\log_2 |x|) \leq O(m\log_2 m)+4|x|+O(1)$. Note that by definition
		of $m$, $4|x|^{2e} < 2^m$, hence $2+2e \log_2 |x| < m$. This implies $m-2e \log_2 |x| > 2$, hence
		$2l' \leq O(m \log_2 m)+4|x|+O(1)$ which means $l \leq O(4|x|2e \log_2(4|x|^{2e}))+2|x|+O(1)$ and this completes the proof.
	\end{proof}
	It is worth mentioning that the forcing notion that was used in \cite{bfkrv10} is a finite condition forcing,
	but the forcing notion ${\cal P}_K$ permits us to have conditions with an infinite domain.
	Note that we essentially use this property of ${\cal P}_K$ in our construction. We do not know whether
	(nonuniform) p-optimal proof systems for $\taut$ exist relative to the original oracle that
	defined in \cite{bfkrv10}. Note that the existence of oracles $\V$ and $\W$ imply several separations
	between conjectures of figure \ref{fig}. The following corollary shows several independence results (not all
	of the separations) of the conjectures of the branches in figure \ref{fig}.
	\begin{corollary}
		Define the following sets:
		\begin{enumerate}
			\item $A=\{\con,\conN\}$,
			\item $B=\{\sat_c,\tfnp_c,\disconp_c\}$.
		\end{enumerate}
		Then for every conjecture $Q \in A$ and every conjecture $Q' \in B$, $Q$ and $Q'$ do not imply each
		other in relativized worlds.
	\end{corollary}
	\begin{proof}
		The corollary follows from theorems \ref{t5.1} and \ref{t5.2}.
	\end{proof}
	\subsection*{Acknowledgment}
	We are indebted to Pavel Pudl\'ak for many invaluable discussions that we have had about this work. We are also grateful to him for his careful readings of the drafts of this paper, his useful comments and suggestions about it, and also pointing out many small errors which led to improvements in its presentation. Additionally, We are grateful to Moritz M\"{u}ller for his careful reading of the draft of this paper and for his useful suggestions. We also thank Michael Rathjen for answering our question about realizability and introducing reference \cite{rat06} to us. This research was partially supported by the ERC Advanced Grant
	339691 (FEALORA).

\end{document}